\documentclass[12pt]{amsart}

\usepackage{graphicx}
\usepackage{amsfonts}
\usepackage{epsf}
\usepackage{amssymb}
\usepackage{amsmath}
\usepackage{amscd}
\usepackage{hyperref}

\newtheorem{theorem}{Theorem}[section]
\newtheorem{proposition}[theorem]{Proposition}
\newtheorem{lemma}[theorem]{Lemma}

\newtheorem{corollary}[theorem]{Corollary}

\theoremstyle{remark} 
\newtheorem{definition}[theorem]{Definition}
\newtheorem{example}[theorem]{Example}
\newtheorem{remark}[theorem]{Remark}

\newcommand{\q}{\mathbb Q}
\newcommand{\z}{\mathbb Z}
\newcommand{\f}{\mathbb{F}}
\newcommand{\lr}{\langle \cdot , \cdot \rangle}

\input xy
\xyoption{all}

\begin{document}

\title[Signatures of certain even symmetric forms]{The signature of an even symmetric form with vanishing associated linking form}
\begin{abstract}
We prove that the signature of an even, symmetric form on a finite rank integral lattice, has signature divisible by $8$, provided its associated linking form vanishes in the Witt group of linking forms. Our result generalizes the well know fact that an even, unimodular form has signature divisible by $8$. 

We give applications to signatures of $4n$-dimensional manifolds, signatures of classical knots, and  provide new restrictions to solutions of certain Diophantine equations.  
\end{abstract}
\author{Stanislav Jabuka}
\email{jabuka@unr.edu}
\address{Department of Mathematics and Statistics, University of Nevada, Reno NV 89557.}
\thanks{This work was generously supported by a research grant from the University of Nevada, Reno.}
\maketitle
\section{Introduction}
\subsection{Background and motivation} 
Symmetric forms on integral lattices are ubiquitous in mathematics, appearing in research areas as diverse as knot theory \cite{Rolfsen}, manifold theory \cite{GompfStipsicz}, group theory \cite{ConwaySloane}, Lie algebras \cite{Humphreys}, elliptic curves \cite{Husemoller}, number theory \cite{Hasse}, and others.  Depending on context, different notions of equivalence between symmetric forms are studied. For some notions of equivalence, complete sets of invariants are known while for others only partial invariants exist. Regardless of the context, one of the most basic invariants of a symmetric form is its signature.  Forms with signature equal to plus or minus their rank play a special role in many areas of mathematics, so do forms with non-vanishing signature, etc. 

It is basic question, and yet not an easy one in general, to try to deduce particulars about the signature of a symmetric form from it's  various properties, or lack thereof. An illustration of this principle is given by the  classical example \cite{Milnor}: 
\begin{theorem} \label{ExampleOfUnimodularEvenForm}
The signature of a unimodular even symmetric form is a multiple of $8$. 
\end{theorem}
In the present work, we aim to generalize this result to a broader class of symmetric forms. We still only consider even symmetric forms, but we shall relax the unimodularity condition. Instead, we will ask that the associated linking form of the given symmetric form be vanishing in a certain Witt ring. Before stating our results, we set the stage by giving our  terminology precise meaning.  
\begin{definition} \label{DefintionOfForms}
Let $V$ be a finite dimensional rational vector space. 
\begin{itemize}
\item[(i)] A {\em (rational) symmetric form} on $V$ is a symmetric, bilinear, non-degenerate map $b:V\times V\to \mathbb Q$. We shall refer to the pair $(V,b)$ or $b$ itself, as a (rational) symmetric form. 
\item[(ii)] An {\em integral lattice $L$ in a rational vector space $V$} is a free Abelian subgroup $L$ of $V$ such that $L\otimes _{\z} \q = V$. An {\em integral lattice $L$} is an integral lattice in some rational vector space $V$.  

A {\em (integral) symmetric form} on an integral lattice $L$ is a symmetric, bilinear, non-degenerate map $b:L\times L\to\z$. We shall refer to the pair $(L,b)$ or to $b$ itself, as a (integral) symmetric form. 
\item[(iii)] A {\em linking form} is a pair $(G,\lambda )$ with $G$ a finite Abelian group and with $\lambda :G\times G \to \q/\z$ a non-degenerate, symmetric, bilinear form. A linking form $(G,\lambda)$ is called {\em metabolic} if there exists a subgroup $H$ of $G$ with $|H|^2=|G|$ and $\lambda |_{H\times H} \equiv 0$.
\end{itemize}
\end{definition}
%
%

\begin{remark} \label{RemarkAboutEvenAndOdd}
If $b:L\times L\to \z$ is a symmetric form on the lattice $L\subseteq V$, then $b$ induces a symmetric form  $b\otimes \text{id}:V\times V\to \q$ with $(b\otimes \text{id})(v\otimes r, w\otimes s) = b(v,w)rs$, where $v,w\in L$, $r,s\in\q$. However, a symmetric form $b:V\times V\to \q$ does not in general induce a symmetric from on a lattice $L\subseteq V$ by restriction, this only happens if Im$(b|_{L\times L}) \subseteq \z$. Given the latter condition, we define the {\em dual lattice $L^\#$ of $L$ in $V$} as the integral lattice $L^\# = \{v\in V \, |\, b(u,v) \in \z, \, \text{ for all }u\in L\},$ and note that $L$ is a subgroup of $L^\#$.
\end{remark}

An integral symmetric form $(L,b)$ is called {\em even} if $b(u,u)$ is an even integer for all $u\in L$, else $L$ is called {\em odd}.  The {\em signature} and {\em determinant} of a rational or integral symmetric form are the signature and determinant of any of its matrix representatives. The determinant of a rational form is only well defined up to multiplication by elements in $\dot \q^2$. 
  
The Witt group $W(\q)$ of the rationals and the Witt group $W(\q/\z)$ of linking forms on finite Abelian groups, will be introduced in Section \ref{SectionOnBackground}. They are both infinite Abelian groups of which the former consists of equivalence classes of rational symmetric forms $(V,b)$, while the latter is comprised of equivalence classes of linking forms $(G,\lambda)$. The nature of the equivalence relations is recounted in Section \ref{SectionOnBackground}. These two groups fit into an exact sequence 

\centerline{
\xymatrix{
0 \ar[r] &\z \ar[rr]^{\iota \phantom{iii}} & & W(\q)  \ar[rr]^{\partial} &  &W(\q/\z) \ar[r] & 0, \\
}
}
\vskip1mm
\noindent for which full details are provided in Section \ref{SectionOnBackground}. For now it suffices to say that $\partial (V,b)$ is the linking form $(L^\#/L,\lambda _b)$ (with $L^\#$ as in Remark \ref{RemarkAboutEvenAndOdd}) associated to any integral lattice $L\subseteq V$, and with $\lambda _b$ given by 
$$\lambda _b:(L^\#/L)\times (L^\#/L) \to \q/\z, \quad \lambda (u+L,v+L) = b(u,v) +\z.$$
Finiteness of the group $L^\#/L$ follows from the non-degeneracy assumption on $b$.
\begin{definition}  \label{AssociatedLinkingFormDefinition}
Let $(L,b)$ be an integral symmetric form and set $V = L\otimes _\z\q$, so that $(V,b\otimes \text{id})$ is a rational symmetric form. Then  $\partial (V,b\otimes \text{id}) = (L^\#/L,\lambda _b)$ is called the {\em linking form associated to the integral symmetric form $(L,b)$}. 
\end{definition}
With these preliminaries understood, we state our main result.  
\begin{theorem} \label{main}
Let $(L,b)$ be an even integral symmetric form of odd determinant. If its associated linking form (Definition \ref{AssociatedLinkingFormDefinition}) vanishes  in the Witt group $W(\q/\z)$,  then the signature of $b$ is a multiple of $8$. 
\end{theorem}
Many examples of integral symmetric forms meeting the hypothesis of Theorem \ref{main} exist, we list several below. If $b$ is even and unimodular, the associated linking form of $(L,b)$ automatically vanishes in $W(\q/\z)$ (see \cite{Milnor}), showing that Theorem \ref{main} generalizes Theorem \ref{ExampleOfUnimodularEvenForm}.  
\subsection{Applications}
Signatures of symmetric forms on integral lattices appear in many different contexts, giving Theorem \ref{main} a wide spectrum of applicability. The results presented in this section are sample applications among many possible. We focus on three areas: Signatures of $4n$-dimensional manifolds, knot signatures and Diophantine equations.  
\subsubsection{Intersection forms of $4n$-dimensional manifolds} 
Let $X$ be a $4n$-dimensional, oriented, compact manifold $X$ with $Y=\partial X$ and assume that $Y$ is a rational homology $(4n-1)$-sphere. The {\em intersection form $Q_X$} is the symmetric, bilinear, non-degenerate form on $H^{2n}(X;\z)/Tor$ defined via the cup product, followed by evaluation on the fundamental class $[X]\in H_{4n}(X;\z)$ of $X$: 
$$Q_X:(H^2(X;\z)/Tor)\times (H^2(X;\z)/Tor)\to \z, \quad Q_X(a,b) = (a\smile b)[X].$$
The pair $(H^2(X;\z)/Tor,Q_X)$ is an integral symmetric form whose signature we denote  by $\sigma (X)$ and refer to as the {\em signature of $X$}. 

The {\em linking form $\lambda _Y$} of the $(4n-1)$-dimensional rational homology sphere $Y$ is the symmetric, bilinear linking form on $H^{2n}(Y;\z)$ defined through the cup product on $Y$, followed again by evaluation on the fundamental class $[Y]\in H_{4n-1}(Y;\z)$:
$$\lambda _Y:H^{2n}(Y;\z)\times H^{2n}(Y;\z) \to \q/\z, \quad \lambda _Y(\alpha, \beta) = \frac{1}{m} (\alpha \smile \sigma)[Y] + \z.$$
Here $m\in \mathbb N$ is such that $m\cdot \beta = 0 \in H^{2n}(Y;\z)$ and $\sigma$ is a $(2n-1)$-cochain on $Y$ with $\delta \sigma = m\cdot \beta$.  With this understood, we have the following consequence of Theorem \ref{main}.  
\begin{proposition} \label{SignatureOf4nManifolds}
Let $X$ be $4n$-dimensional, oriented, compact manifold with boundary $Y$ a rational homology $(4n-1)$-sphere, and assume that its intersection form $Q_X$ is even. If $|H^{2n}(Y;\z)|$ is odd and the linking form $\lambda _Y$ of $Y$ is metabolic, then $\sigma (X) \equiv 0\, (\text{mod } 8)$.   
\end{proposition}
We note that if $Y$ is an integral homology sphere, then $Q_X$ is unimodular and the proposition follows from Theorem \ref{ExampleOfUnimodularEvenForm}. 

By way of comparison, recall that Rokhlin's theorem \cite{BlaineMichelson} posits that the signature $\sigma(X)$ of a $4$-dimensional, smooth, oriented, closed, spin manifold $X$ is divisible by $16$. If $X$ is spin then its intersection form is even, and if $X$ is also simply-connected, then it is spin if and only if its intersection form is even. While the conclusions of Proposition \ref{SignatureOf4nManifolds} are substantially weaker than Rokhlin's, so are its hypotheses. Most notably, Proposition \ref{SignatureOf4nManifolds} does not require the smoothness condition, and it applies to manifolds of all dimensions that are a multiple of $4$. 

\begin{example} \label{ExampleOfLensSpaces}
If $Y$ is the $3$-manifold obtained by $p^2/q$-framed Dehn surgery on a knot $K\subseteq S^3$, with $\gcd (p,q)=1$ and $p$ odd, then $(H^2(Y;\z),\lambda _Y)$ is metabolic. Consequently, any $4$-manifold $X$ with even intersection form $Q_X$ and with boundary $Y$, has signature a multiple of $8$. 

Many manifolds $Y$ obtained in this manner do not bound rational homology balls. Examples include lens spaces $L(9,q)$ with $q=1,8$, $L(25,q)$ with $q=1,2,3,8,12,...$, $L(49,q)$ with $=1,2,3,4,5,9,10,11,12,\dots $, etc. Accompanying $4$-manifolds $X$  with $\partial X=Y$ and $Q_X$ even, can also be found explicitly.  For instance, the boundary of the linear plumbing on $8$ vertices, each with weight $-2$, is the lens space $L(9,1)$. The signature of this $4$-manifold is $-8$. 
\end{example}
\subsubsection{Knot signatures}
In this section we consider {\em classical knots}, that is isotopy classes of smooth embeddings of $S^1$ into $S^3$. Associated to such a knot $K$, along with the choice of a Seifert surface $\Sigma$, is its {\em Seifert form} $\xi=\xi_{K,\Sigma} :H_1(\Sigma ;\z)\times H_1(\Sigma ;\z)\to \z$. The pair $\left( H_1(\Sigma;\z),\xi+\bar \xi\,\right)$ is an integral symmetric form, where $\bar \xi (a,b) = \xi(b,a)$.  
The {\em signature $\sigma (K)$ of the knot $K$} is the signature of $\xi+\bar \xi$, and it is not hard to see that this quantity is independent of all choices. The {\em determinant $\det K$ of the knot $K$} is defined as $\det K = \det (\xi +\bar \xi)$, which is always an odd integer, and well defined up to sign. For a more detailed treatment of these concepts, see for instance \cite{Rolfsen}. 
\begin{proposition} \label{CorollaryAboutKnotSignatures}
Let $K$ be knot in $S^3$, let $\Sigma$ be a Seifert surface for $K$ and let $\xi=\xi_{K,\Sigma}$ be its associated Seifert form. If $\partial (\xi+\bar \xi)=0\in W(\q/\z)$ then $\sigma (K)$ is a multiple of $8$. 
\end{proposition}
To put this result into context, we note that according to Theorem 5.6 in \cite{Murasugi}, the modulus of $\sigma (K)$ with respect to $4$ is determined by its determinant in that 
\begin{equation} \label{SignatureCongruencesDependingOnDeterminant}
\sigma (K) \equiv \left\{
\begin{array}{cl}
0 \, (\text{mod } 4) & \quad ; \quad |\det K| \equiv 1 \, (\text{mod } 4), \cr
2 \, (\text{mod } 4) & \quad ; \quad |\det K| \equiv 3 \, (\text{mod } 4).
\end{array}
\right.
\end{equation} 
All knots $K$ with $\partial (\xi+\bar \xi)=0$ have, up to sign, a square odd determinant, and thus satisfy the congruence $|\det K| \equiv 1 \, (\text{mod } 4)$. Accordingly, Proposition \ref{CorollaryAboutKnotSignatures} can be seen as a refinement of the first line in \eqref{SignatureCongruencesDependingOnDeterminant} (for those knots $K$ with $\partial (\xi +\bar \xi)=0$). 

\begin{example} \label{LowCrossingKnots}
Instances of low crossing knots that satisfy the hypothesis of Proposition \ref{CorollaryAboutKnotSignatures} are given by $9_1$, $11a_{263}$, $11a_{334}$, $11a_{364}$ and $12a_{0093}$ (our notation referring to that from KnotInfo \cite{KnotInfo}). 

Additional examples can be obtained by considering connected sums $K_1\# K_2$ of knots $K_1, K_2$ in which both $K_1$ and $K_2$ fail the hypothesis of Theorem \ref{main}, but with $K_1\# K_2$ satisfying it. Examples of this kind are given by  $7_5\#8_2$, $6_3\#8_1$. 

All of these examples have signature $-8$ except $6_3\#8_1$ which has signature $0$. 
\end{example}
\subsubsection{Diophantine equations} \label{SubsectioinOnDiophantineEquations}
The applications presented in this section are inspired by the work of Stoimenov \cite{Stoimenow}. Stoimenow considers knots $K = K(p_1,...,p_n)$ that depend on a number of integral parameters $p_1,...,p_n\in \mathbb Z$, and asks what can be said about the solutions of the equation $\det K = \pm 1$? The determinant of $K$ can be computed from the parameters $p_1,...,p_n$, typically as a polynomial function, turning the relation $\det K=\pm 1$ into a Diophantine equation in the unknowns $p_1,..,p_n$. 

If the determinant of $K$ is $\pm 1$, then the symmetric Seifert form $\xi +\bar \xi$ of $K$ is unimodular and even, and as such has signature a multiple of 8 (cf. Theorem \ref{ExampleOfUnimodularEvenForm}). Computing the signature $\sigma (K)$ in terms of $p_1,...,p_n$ and evaluating the congruence $\sigma (K) \equiv 0\, (\text{mod } 8)$ gives restriction on the parameters $p_1,...,p_n$, and hence information about the solution set of the Diophantine equation $\det K =\pm 1$. 

We take the same approach here but consider the more general equation $\det K = \pm m^2$, with $m$ an odd integer. We can still conclude that $\sigma (K) \equiv 0\, (\text{mod } 8)$ provided that $\partial (\xi +\bar \xi) = 0\in W(\q/\z)$. While the Diophantine equations we consider are more general than those from \cite{Stoimenow}, the families of knots we can use are more restrictive, owing to the condition $\partial (\xi +\bar \xi) = 0$. Nevertheless, we are able to present an example where the latter condition is implied automatically by the equation $\det K = -m^2$ itself. Namely, consider the pretzel knot $K=P(p,q,r)$ with $p,q$ odd and $r$ an even integer as defined in Figure \ref{pic1}. 
\begin{figure}[htb!] 
\centering
\includegraphics[width=7cm]{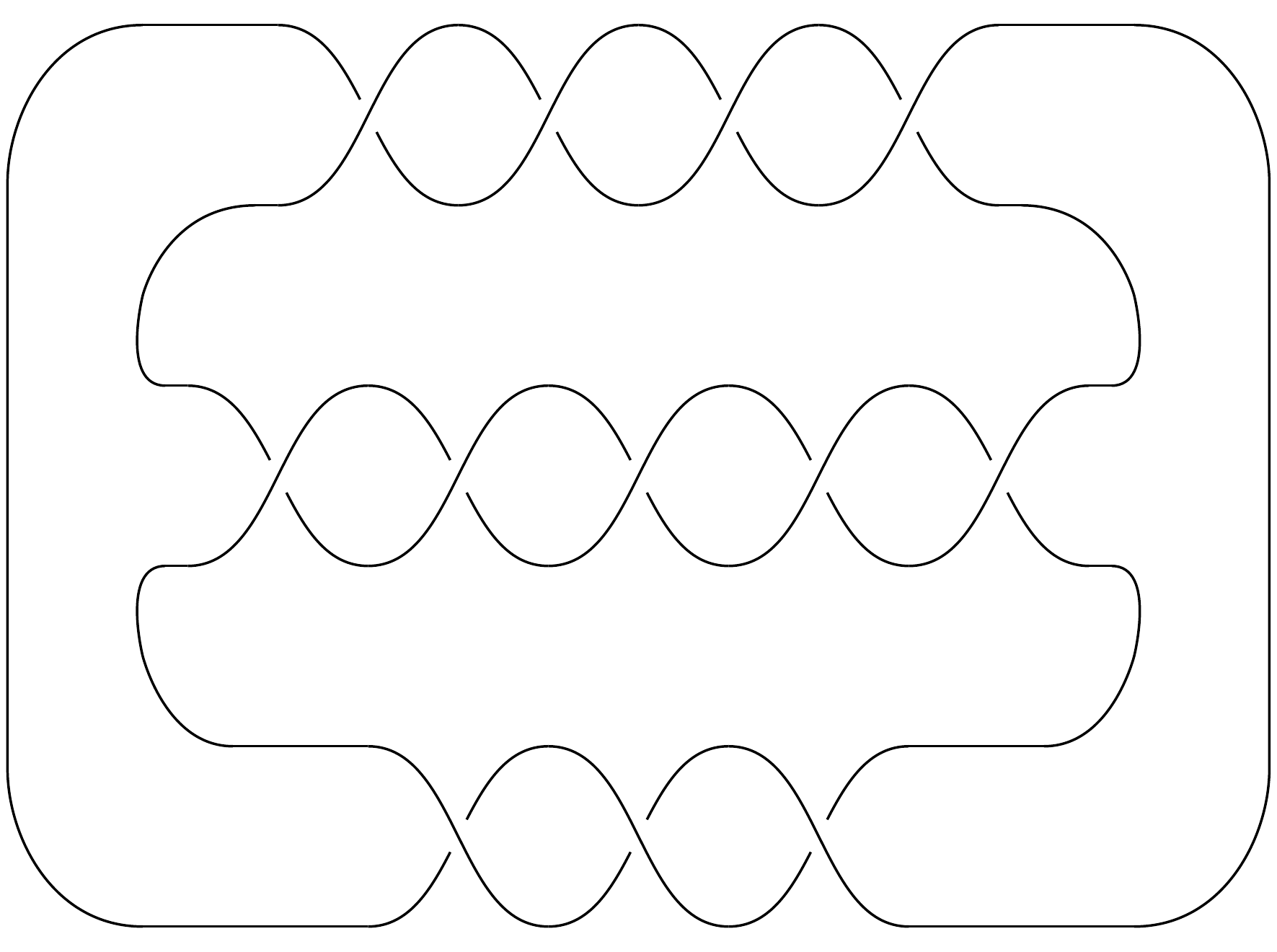}
\caption{The pretzel knot $P(p,q,r)$ is formed by starting with $3$ pairs of parallel strands, placing $p$, $q$ and $r$ half-twists into them, and connecting their ends as shown in the picture. The half-twists are right-handed if associated to a positive integer, else they are left-handed. The picture shows the knot $P(3,-5,-4)$. }  \label{pic1}
\end{figure}
The equation $\det K = -m^2$ becomes 
\begin{equation}  \label{DiophantineEquation}
pq+pr+qr = -m^2. 
\end{equation}   
Examining the reduction of equation \eqref{DiophantineEquation} modulo $8$, it is easy to see that $p+q$ must be congruent to either $0$ or $4$ modulo $8$. An application of Theorem \ref{main} leads to the following strengthened conclusion. 
\begin{proposition} \label{DiophantineTheorem}
If $p,q,m $ are odd integers and $r$ is an even integer that solve the Diophantine equation $pq+pr+qr = -m^2$, then $ p+q \equiv 0\, (\text{mod } 8)$.
\end{proposition}
To underscore the subtlety of this result, consider the slight modification of equation \eqref{DiophantineEquation} given by 
\begin{equation} \label{DiophantineEquation2}
pq+pr+qr = m^2. 
\end{equation}
A reduction argument shows that $p+q$ must now be congruent to either $2$ or $6$ modulo 8. However, no additional restriction akin to that from Proposition \ref{DiophantineTheorem} exist in this setting, as is clear from the table of examples below:  
\begin{center}
\begin{tabular}[t]{|c|c||c|} \hline
$\mathbf{(p,q,r)}$ &  $\mathbf m$ & $\mathbf{p+q \,\, (\text{\bf mod } 8)}$  \cr \hline \hline
$(3,7,6)$ &  $9$ & $2$  \cr \hline
$(-9,3,-6)$ &  $3$ & $2$  \cr \hline
$(-7,-3,-10)$ &  $11$ & $6$  \cr \hline
$(3,-5,-8)$ &  $1$ & $6$  \cr \hline
\end{tabular}
\end{center}

The reason for this is that equation \eqref{DiophantineEquation} implies the vanishing of $\partial (\xi +\bar \xi)$ in the Witt group of linking forms, while equation \eqref{DiophantineEquation2} does not. Accordingly, equation \eqref{DiophantineEquation} imposes  restrictions on $p+q$ stemming from the congruence $\sigma (K) \equiv 0\, (\text{mod } 8)$ from Theorem \ref{main}, while no such restrictions exist in the case of equation  \eqref{DiophantineEquation2}.
\subsection{Acknowledgements} I am indebted to Andrew Ranicki, Paolo Lisca and Alexis Marin for helpful comments on an earlier version of this work. 
\section{Background material} \label{SectionOnBackground}
\subsection{Witt groups}
In this section we introduce the Witt group $W(\f)$ associated to a field $\f$, and the Witt group $W(\q/\z)$ of linking forms on finite Abelian groups. Both groups carry a natural ring structure but as we shall not need it in the sequel, we omit it from our treatment. Special emphasis is given to the case of $\f=\q$. Our exposition draws from  \cite{Alexander, Lam}. 

Throughout, let $\f$ denote a field. We shall give a unified treatment of Witt groups over fields and the Witt group of linking forms on finite Abelian groups. Thus, let $(M,\xi)$ be either a pair $(V,b)$ consisting of a finite dimensional $\f$-vector space $V$ and a symmetric, non-degenerate form $b:V\times V\to \f$, or a pair $(G,\lambda)$ consisting of a finite Abelian group $G$ and a symmetric, non-degenerate form (referred to as a linking form) $\lambda:G\times G\to \q/\z$. Given a subset $N\subseteq M$ we define its {\em orthogonal complement $N^\perp$} as 
$$N^\perp = \{x\in M\, |\, \xi(x,y)=0 \text{ for all } y\in N\}.$$
The  form $(M,\xi)$ is called {\em metabolic} if there exists a subspace or a subgroup $N\subseteq M$ (according to whether $M=V$ or $M=G$) such that $N=N^\perp$. In this case, we refer to $N$ as a {\em metabolizer} of $(M,\xi)$. We leave it as an exercise for the reader to verify that this notion of metabolic is equivalent to that given in part (iii) of Definition \ref{DefintionOfForms} for linking forms. 

If $(M,\xi) = (V,b)$ is metabolic,  then $(V,b)$ has  determinant a square and, if $\f=\q$, it also has signature zero. Similarly (as already implied by part (iii) of Definition \ref{DefintionOfForms}) if $(M,\xi) = (G,\lambda)$ is metabolic then $|G|$ is a square.  

Two forms $(M_1,\xi_1)$ and $(M_2,\xi_2)$ are called {\em isomorphic}  if there exists an isomorphism $\varphi :M_1\to M_2$ such that $\xi_2(\varphi (x),\varphi(y)) = \xi_1(x,y)$ for all $x,y\in M_1$. We define the direct sum $(M_1,\xi_1)\oplus(M_2,\xi_2)$ of two forms $(M_1,\xi_1)$ and $(M_2,\xi_2)$ as 
$$ (M_1,\xi_1)\oplus(M_2,\xi_2) = (M_1\oplus M_2,\xi_1+\xi_2).$$
The forms $(M_1,\xi_1)$ and $(M_2,\xi_2)$ are called {\em algebraically concordant} if there exist metabolic forms $(N_1,\eta_1)$ and $(N_2,\eta_2)$ such that $(M_1,\xi_1)\oplus (N_1,\eta_1)$ is isomorphic to $(M_2,\xi_2)\oplus (N_2,\eta_2)$. It is not hard to verify that algebraic concordance is an equivalence relation, and the set $W(R)$ of its equivalence classes (with $R=\f$ in the case of forms on vector spaces, and $R=\q/\z$ in the case of linking forms) is an Abelian group under the direct sum operation.  We shall refer to 
\begin{itemize}
\item $W(\f)$ as the {\em Witt group of the field $\f$},
\item $W(\q/\z)$ as the {\em Witt group of linking forms}.
\end{itemize}
The zero element of $W(R)$ is given by the equivalence class of any metabolic form $(M,\xi)$ and the inverse of $(M,\xi)$ is given by $(M,-\xi)$. 
The isomorphism types of the Witt groups $W(\f)$ are known for many fields $\f$, the isomorphism type of $W(\q/\z)$ is also well understood. The next theorem can be found in \cite{Alexander, Milnor}. 
\begin{theorem}
The isomorphism types of $W(\q)$ and $W(\q/\z)$ are given by 
$$W(\q) \cong \mathbb Z\oplus \mathbb Z_2^\infty \oplus \mathbb Z_4^\infty\quad \text{ and } \quad W(\q/\z) \cong \mathbb Z_2^\infty \oplus \mathbb Z_4^\infty.$$
In the above, $\mathbb Z_p^\infty$ denotes the countably infinite direct sum of $\z_p:=\mathbb Z/p\z$. 
\end{theorem}
It is hard not to notice the similarity between the two groups $W(\q)$ and $W(\q/\z)$. Indeed this similarity is deeply rooted as we proceed to explain below.

Let $(V,b)$  be a rational form and $L\subseteq V$ an integral lattice. Going forward we shall only consider integral lattices for which $b(u,v) \in \z$ for all $u,v\in L$, an assumption we shall rely on tacitly. Recall that the dual lattice $L^{\#}$ of $L$ was defined as 
$$L^\# = \{v\in V\, |\, b(u,v) \in \z \text{ for all }u\in L\}.$$
Clearly $L$ is contained in $L^\#$ and by the non-degeneracy of $b$, the quotient group $L^\#/L$ is a finite group with cardinality $|\det b\,|$. Indeed, if we pick a basis $\{e_1,..,e_n\}$ of $L$ and let $[b_{i,j}]$ be a matrix representative for $b$ (so that $b_{i,j}=b(e_i,e_j)$), then $L^\#$ is the span inside of $V$ of the column vectors of the matrix $[c_{i,j}]:=[b_{i,j}]^{-1}$, that is 
$$L^\# = Span \{f_1,...,f_n\}, \quad \text{ with } \quad f_k = \sum _{i=1}^n c_{k,i} \,e_i. $$
%
\begin{remark} \label{RemarkAboutDeterminant}
In light of this characterization of $L^\#$, we point out for later use that $\det b\cdot u \in L$ for every $u\in L^\#$. 
\end{remark}
We extend $b$ to a form (of the same name) $b:L^\#\times L^\#\to \mathbb Q$ by linearity, and define $\lambda _b:(L^\#/L)\times (L^\#/L)\to \q/\z$ by setting 
$$\lambda _b(u+L,v+L) = b(u,v)+\z.$$  
Note that $\lambda _b$ is well defined as $b(u,v)$ is an integer whenever $u,v\in L$. 
\begin{theorem} \label{TheoremAboutExactSequence}
The assignment $(V,b)\mapsto (L^\#/L,\lambda _b)$, where $L\subseteq V$ is any integral lattice in $V$, 
induces a group homomorphism $\partial :W(\q) \to W(\q/\z)$ that fits into a split exact sequence 
\begin{equation}\label{TheExactSequence}
\xymatrix{
0 \ar[r] &\z \ar[rr]^{\iota \phantom{iii}} & & W(\q)    \ar@/_2pc/[ll]_\sigma \ar[rr]^{\partial \phantom{iii}} &  &W(\q/\z) \ar[r] & 0. \\
}
\end{equation}
\vskip1mm
The map $\iota$ is given by $\iota (n) = (\q^{|n|}, Sign(n)\cdot I_{|n|})$ where $I_{|n|}$ is the form $I_{|n|}(x,y) = x_1y_1+ \dots +x_{|n|}y_{|n|}$. A splitting map $\sigma :W(\q)\to \z$ is given by the signature function.
\end{theorem}
%

\subsection{A computational algorithm}
To use Theorem \ref{main} effectively, it is important to be able to check the condition $\partial (V,b)=0$ for a rational form $(V,b)$. We describe in this section an explicit algorithm for doing so. It relies on an presentation of the Witt group $W(\q)$, an understanding of the Witt groups of the finite fields $\f_\wp$, and a relation of the latter to the structure of $W(\q)$.  

Let $\f$ be a field of characteristic not equal to $2$, and let $b:V\times V\to \f$ be a symmetric, non-degenerate form on the finite dimensional $\f$-vector space $V$. By a suitable choice of basis for $V$, the form $b$ can be diagonalized and becomes a direct sum of $1$-dimensional forms $\langle a_i\rangle$, $a_i\in \dot \f=\f-\{0\}$, $i=1,...,\dim _\f V$. Here $\langle a \rangle$ with $a\in \dot \f$ is short for the $1$-dimensional form $(\f, \langle a \rangle)$ determined by $\langle a\rangle (1,1) = a$. Thus $\{\langle a\rangle \, |\, a\in \dot \f\}$ is a generating set for the Witt group $W(\f)$. A set of relations R1--R3, giving a presentation of $W(\f)$ with the above generating set, is given by 
\begin{equation} \label{WittGroupRelations}
\begin{array}{crll}
R1: &\quad   \langle -a\rangle &\hspace{-2mm} = - \langle a\rangle,  \quad  & \forall a\in \f, \cr
R2: &\quad  \langle a \cdot b^2\rangle  &\hspace{-2mm} = \langle a \rangle, \quad & \forall a, b \in \dot \f , \cr
R3: &\quad  \langle a\rangle \oplus \langle b \rangle &\hspace{-2mm}  = \langle a+b\rangle \oplus \langle ab(a+b)\rangle, \quad & \forall a,b \in \f \text{ with } a+b\ne 0. \cr 
\end{array}
\end{equation}
The necessity of these relations is easy to establish, the fact that they are also sufficient can be found in \cite{Milnor}. While symmetric forms over fields of characteristic $2$ cannot always be diagonalized, they are algebraically concordant to diagonal forms, and so the same presentation remains valid over such fields as well. 

Let $\f_\wp$ denote the finite field of prime order $\wp$. The Witt groups $W(\f_\wp)$  are well understood and are given by 
\begin{equation} \label{FiniteFieldsWittGroups}
W(\f_\wp) \cong \left\{
\begin{array}{cl}
\mathbb Z_2 & \quad ; \quad \wp =2 , \cr
\mathbb Z_2\oplus \mathbb Z_2  & \quad ; \quad \wp \equiv 1 \, (\text{mod } 4) , \cr
\mathbb Z_4  & \quad ; \quad \wp \equiv 3 \, (\text{mod } 4). \cr
\end{array}
\right.
\end{equation}
Generators of $W(\f_\wp)$ are given by $\langle 1 \rangle$ in the case of $\wp=2$ or $\wp \equiv 3\, (\text{mod } 4)$, and by $\langle 1\rangle$, $\langle b\rangle$ in the case of $\wp \equiv 1\, (\text{mod } 4)$. Here $b$ is any non square in $\dot \f_\wp$, that is $b\in \dot \f _\wp - \dot \f ^2_\wp$. 

Our reason for mentioning the Witt groups of the fields $\f_\wp$ is that their direct sum is isomorphic, in a rather natural way, to the torsion subgroup of $W(\q)$. Namely, for a prime $\wp$ let us first define the function $\partial _\wp :W(\q)\to W(\f_\wp)$ by defining it on a generator $\langle \frac{a}{b} \cdot \wp^n \rangle $ with $a,b \in \dot\z$  relatively prime to $\wp$, and with $n\in \z$: 
\begin{equation*}
\partial _\wp \left( \left\langle \frac{a}{b} \cdot \wp ^n\right\rangle \right) = \left\{
\begin{array}{cl}
\langle \bar a \bar b \rangle & \quad ; \quad n \text{ is odd}, \cr & \cr
0 & \quad ; \quad n \text{ is even}.
\end{array}
\right.  
\end{equation*}
Here $\bar a, \bar b \in \f_\wp$ are the mod $\wp$ reductions of $a,b\in \mathbb Z$. The direct sum $\oplus _\wp \partial _\wp$, taken over all prime number $\wp$, gives a homomorphism from $W(\q)$ to $\oplus_\wp W(\f_\wp)$, the latter being isomorphic to $\mathbb Z_2^\infty \oplus \mathbb Z_4^\infty$, the torsion subgroup of $W(\q)$. 
\begin{theorem} \label{ConditionChecker}
The sequence 
$$0\to \z\stackrel{\iota}{\longrightarrow} W(\q) \stackrel{\oplus \partial _\wp}{\longrightarrow} \oplus _\wp W(\f_\wp) \to 0$$
is split exact, with $\oplus _\wp \partial _\wp$ restricting to an isomorphism from the torsion subgroup of $W(\q)$ to $\oplus_\wp W(\f_\wp)\cong \mathbb Z_2^\infty \oplus \mathbb Z_4^\infty$. For a rational form $(V,b)\in W(\q)$, the equality $\partial (V,b)=0$ holds if and only if $\partial _\wp (V,b)=0$ for every prime $\wp$. 
\end{theorem}
\subsection{Metabolizers of rational and linking forms}

We finish this section with a comparison of the metabolizer of a rational form $(V,b)$ and its associated linking form $\partial (V,b)$.
\begin{lemma} \label{AboutOrthogonalComplements}
Let $(V,b)$ be a rational form and let $\partial (V,b) = (L^\#/L,\lambda _b)$ for some integral lattice $L$ in $V$. If $N\subseteq L^\#/L$ is a metabolizer with $N=L_1/L$ for some subgroup $L_1\subseteq L^\#$, then $L_1$ is also an integral lattice in $V$ and $L_1^\# = L_1$. If additionally $\det b|_{L\times L}$ is odd, then $(L_1,b|_{L_1\times L_1})$ is an even form. 
\end{lemma} 
\begin{proof}
Note that by virtue of $N$ being a metabolizer, we obtain that $\lambda _b(u+L,v+L)$ is an integer for any choice of $u,v\in L_1$. Thus $(L_1,b)$ is an integral lattice in $V$ and so its dual lattice $L_1^\#$ is well defined, and given by 
$$L_1^\# = \{ v\in L_1\, |\, b(u,v)\in \z \text{ for all } u\in L_1\}.$$
Clearly $L_1$ is contained in $L_1^\#$. To obtain the opposite inclusion, let $v\in L_1^\#$. Then $b(u,v)\in \z$ for all $u\in L_1$ so that $\lambda _b(u+L,v+L) = 0$ for all $u+L\in L_1/L=N$. Thus, $v+L\in N^\perp = N = L_1/L$ showing that $v\in L_1$. The equality $L_1=L_1^\#$ follows. 

If $\det b|_{L\times L} = m$ for some odd $m\in \mathbb N$, then $m\cdot u \in L$ for any choice of $u\in L_1\subset L^\#$, see Remark \ref{RemarkAboutDeterminant}. Thus $m^2\cdot b(u,u) = b(m\cdot u, m\cdot u) \in 2\mathbb Z$ and since $m$ is odd, then $b(u,u)\in 2\mathbb Z$ for every $u\in L_1$. Thus $(L_1,b|_{L_1\times L_1})$ is an even integral form. 
\end{proof}
\section{Proofs}
\subsection{The proof of Theorem \ref{main}}
Let $(V,b)$ be an even, symmetric, rational form and let $L\subseteq V$ be an integral lattice, such that $\det b|_{L\times L}$ is odd. Let $(L^\#/L,\lambda _b)=\partial (V,b)$ be the associated linking form and assume that $(L^\#/L,\lambda _b)= 0\in W(\q/\z)$. Let $N\subseteq L^\#/L$ be a metabolizer for $(L^\#/L,\lambda _b)$ and let $L_1\subseteq L^\#$ be a subgroup such that $N=L_1/L$. It follows from Lemma \ref{AboutOrthogonalComplements} that $L_1\subseteq V$ is itself an even integral lattice and that $L_1=L_1^\#$. Let $|L^\#/L| = m^2=|\det b\, |$ so that $|N|=m=\sqrt{|\det b\,|}$. 

Following Appendix 4 in \cite{Milnor} (see also \cite{BaillyCabral, BargeLannesLatourVogel}), we introduce the {\em Gauss sum $G(b)$} for an even, symmetric form $b$ on an integral lattice $L\subseteq V$, with the formula:
$$G(b) = \sum _{u\in (L^\#/L)} e^{ \pi i \cdot b(u,u)}\in \mathbb C.$$
The following lemma and theorem can be found in Appendix 4 in \cite{Milnor}.
\begin{lemma} \label{LemmaFromMilnor}
Let $L$ be a sublattice of $L_1$ of index $m$, then $G(L) = m\cdot G(L_1)$.
\end{lemma}
\begin{theorem} \label{TheoremFromMilnor}
Let $(L,b)$ be an even, symmetric, integral lattice. Then 
$$ G(L) = \sqrt{|\det b|} \cdot  e^{2\pi i \cdot  \frac{\sigma(b)}{8}}.$$ 
\end{theorem}
Lemma \ref{LemmaFromMilnor} shows that $G(L) = m\cdot G(L_1)$ while Lemma \ref{AboutOrthogonalComplements} implies that 
$$G(L_1)= \sum _{u\in (L_1^\#/L_1)} e^{\pi i b(u,u)} = \sum _{u\in (L_1/L_1)}e^{\pi i \cdot 0 } = 1.$$ 
Using these two relations in Theorem \ref{TheoremFromMilnor} gives
$$ m \cdot e^{2\pi i \cdot \frac{\sigma(b)}{8}}  = G(L)= m\cdot G(L_1) = m.$$
implying that $e^{2\pi i \cdot \frac{\sigma(b)}{8}}=1$. It follows that $\frac{\sigma (b)}{8}$ is an integer, as claimed in Theorem \ref{main}.
\subsection{The proof of Proposition \ref{SignatureOf4nManifolds}}
Proposition \ref{SignatureOf4nManifolds} follows at once from Theorem \ref{main} with the help of the following result \cite{Alexander, BargeLannesLatourVogel}: 
\begin{theorem}
Let $X$ be a $4n$-dimensional, oriented, compact manifold with boundary the rational homology $(4n-1)$-sphere $Y$. Then the map $\partial :W(\q)\to W(\q/\z)$ from equation \eqref{TheExactSequence} in Theorem \ref{TheoremAboutExactSequence}, sends the rational Witt class of $\left( H_{2n}(X;\q),Q_X\right)$ to minus the Witt class  of the linking form $\left( H^{2n}(Y;\z),\lambda _Y\right)$. 
\end{theorem}
If $Y$ is assumed metabolic, as in Proposition \ref{SignatureOf4nManifolds},  then the Witt class of $\left( H^{2n}(Y;\z),\lambda _Y\right)$ is zero in $(W\q/\z)$ and hence $\partial \left( H_{2n}(X;\q),Q_X\right) = 0$. The determinant of $Q_X$ is given by $|H^{2n}(Y;\z)|$, establishing the applicability of Theorem \ref{main} to $Q_X$. 

If $Y$ is the $3$-manifold resulting from $p^2/q$-framed Dehn surgery on a  knot $K\subset S^3$ with $p$ odd, then $H^2(Y;\z)\cong \mathbb Z_{p^2}$ and the linking form of $Y$ is given by (see Lemma 2 in \cite{Lickorish}): 
$$\lambda _Y(a,b) = \frac{abq}{p^2}, \quad \quad a,b\in \mathbb Z_{p^2}.$$  
Thus a metabolizer is given by $N=p\cdot H^2(Y;\z)$. 

If $Y$ bounds a $4$-manifold $X$ with even intersection form $Q_X$, and $Y$ also bounds a rational homology ball $B$, then $X\cup _YB$ is a closed $4$-manifold, with unimodular intersection form $Q_{X\cup _Y B}$. If $Q_{X\cup _Y B}$ is also even  then it has signature a multiple of $8$ and in this case Proposition \ref{SignatureOf4nManifolds} follows from Theorem \ref{ExampleOfUnimodularEvenForm} since $\sigma (X) = \sigma (X\cup _YB)$.  The fact that the lens spaces from Example \ref{ExampleOfLensSpaces} do not bound rational homology $4$-balls, follows from Lisca's work \cite{Lisca}.
\subsection{The proof of Proposition \ref{CorollaryAboutKnotSignatures}}
Proposition \ref{CorollaryAboutKnotSignatures} follows readily from Theorem \ref{main} by relying on standard facts about the Seifert form $\xi = \xi_{K,\Sigma}$ of a knot $K$, associated to a Seifert surface $\Sigma$, both of which can be found in \cite{Rolfsen}:
\begin{itemize}
\item[(i)] $\xi +\bar \xi$ is symmetric and non-degenerate. 
\item[(ii)] $\xi +\bar \xi$ is even (as is obvious from $(\xi + \bar \xi)(u,u) = 2 \xi (u,u)$). 
\item[(iii)] The determinant of $\xi +\bar \xi$ is always odd. 
\end{itemize}

Concerning Example \ref{LowCrossingKnots}, the quantity $\partial (\xi +\bar\xi)$ for the knots in said example can easily (with computer help) be evaluated explicitly using the algorithm from Theorem \ref{ConditionChecker}. For instance if $K=9_1$, then $\xi+\bar\xi$ can be represented by the matrix $A$ (see \cite{KnotInfo}), which becomes diagonal under a basis change with transition matrix $P$: 
$$P\cdot A\cdot P^\tau = Diag\left( -2,\, -\frac{3}{2},\, -\frac{4}{3},\, -\frac{5}{4},\, -\frac{6}{5},\, -\frac{7}{6}, \,-\frac{8}{7},\, -\frac{9}{8} \right), \quad \text{ with } $$
$$P = 
 {\tiny 
\left[
\begin{array}{rrrrrrrr}
1 & 0 &0 &0 &0 &0 &0 & 0 \cr
-\frac{1}{2} & 1 &0 &0 &0 &0 &0 & 0 \cr
-\frac{1}{3} & -\frac{1}{3} & 1 &0 &0 &0 &0 & 0 \cr
-\frac{1}{4} & -\frac{1}{4} & -\frac{1}{4} & 1 &0 &0 &0 & 0 \cr
-\frac{1}{5} & -\frac{1}{5} & -\frac{1}{5} &-\frac{1}{5} & 1 &0 &0 & 0 \cr
-\frac{1}{6} & -\frac{1}{6} & -\frac{1}{6} &-\frac{1}{6} &-\frac{1}{6} &1 &0 & 0 \cr
-\frac{1}{7} & -\frac{1}{7} & -\frac{1}{7} &-\frac{1}{7} &-\frac{1}{7} &-\frac{1}{7} &1 & 0 \cr
-\frac{1}{8} & -\frac{1}{8} & -\frac{1}{8} &-\frac{1}{8} &-\frac{1}{8} &-\frac{1}{8} &-\frac{1}{8} & 1 
\end{array}
\right] \quad \text{ \normalsize and } \quad 
A= -\left[
\begin{array}{rrrrrrrr}
2 & 1 &1 &1 &1 &1 &1 & 1 \cr
1 & 2 &1 &1 &1 &1 &1 & 1 \cr
1 & 1 &2 &1 &1 &1 &1 & 1 \cr
1 & 1 &1 &2 &1 &1 &1 & 1 \cr
1 & 1 &1 &1 &2 &1 &1 & 1 \cr
1 & 1 &1 &1 &1 &2 &1 & 1 \cr
1 & 1 &1 &1 &1 &1 &2 & 1 \cr
1 & 1 &1 &1 &1 &1 &1 & 2 \cr
\end{array}
\right].
}
$$
The rational Witt class $\varphi(9_1)$ of the knot $9_1$ is then given by 
$$\varphi(9_1) = \langle -2\rangle \oplus \left\langle -\frac{3}{2}\right\rangle \oplus  \left\langle -\frac{4}{3}\right\rangle \oplus  \left\langle -\frac{5}{4}\right\rangle \oplus  \left\langle -\frac{6}{5}\right\rangle \oplus  \left\langle -\frac{7}{6}\right\rangle \oplus  \left\langle -\frac{8}{7}\right\rangle \oplus  \left\langle -\frac{9}{8}\right\rangle.  $$
The only odd primes $\wp$ that occur as factors on the right-hand side above are $\wp = 3,5,7$, showing that $\partial _\wp (\varphi(9_1)) = 0$ if $\wp\ne 3,5,7$. For the latter three primes one obtains 
\begin{align*} 
\partial _3 (\varphi (9_1)) & = 
0\oplus \langle -2\rangle\oplus \langle -4\rangle\oplus 0 \oplus\langle -10\rangle\oplus \langle -14\rangle\oplus 0 \oplus 0,  \cr
& = \langle 1\rangle\oplus \langle -1\rangle \oplus\langle -1\rangle\oplus \langle 1\rangle, \cr
& = 0, \cr 
& \cr
\partial _5 (\varphi (9_1)) & = 0\oplus 0\oplus 0\oplus \langle -4\rangle\oplus \langle -6\rangle\oplus 0 \oplus 0 \oplus 0,  \cr
& =  \langle 1\rangle\oplus \langle -1\rangle, \cr
& = 0, \cr 
& \cr
\partial _7 (\varphi (9_1)) & = 
0\oplus 0\oplus 0\oplus 0\oplus 0\oplus  \langle -6\rangle\oplus \langle -8\rangle\oplus 0,  \cr
& =  \langle 1\rangle\oplus \langle -1\rangle, \cr
& = 0.
\end{align*}
The computations for the other knots in Example \ref{LowCrossingKnots} are similar. 
\subsection{The proof of Proposition \ref{DiophantineTheorem}}
For odd integers $p,q$ and an even integer $r$, consider the pretzel knot $P(p,q,r)$ is defined as in Figure \ref{pic1}. The determinant of $P(p,q,r)$, as computed with respect to a specific Seifert surface, is $\det P(p,q,r) = pq+pr+qr$ (see \cite{jabuka1}) and so the Diophantine equation \eqref{DiophantineEquation} can be restated as $\det P(p,q,r) = -m^2$. The rational Witt class of $P(p,q,r)$, denoted by $\varphi (P(p,q,r))$, is the rational Witt class of $\xi+\bar \xi$, and has been computed in \cite{jabuka2, jabuka1}. Up to summing with terms of the form $\langle \pm 1\rangle$, it is given by 
\begin{equation} \label{WittClassOfPretzelKnot}
\varphi(P(p,q,r)) = \langle p \rangle \oplus \langle q \rangle \oplus \langle r \rangle \oplus \langle pqr \rangle. 
\end{equation}
As the summands $\langle \pm 1\rangle$ lie in the kernel of the map $\partial :W(\q)\to W(\q/\z)$, then can and shall be ignored. Assuming the equation $pq+pr+qr=-m^2$ for some odd integer $m\in \mathbb N$, we need to show that $\partial _\wp $ gives zero when applied to the expression on the right-hand side of equation \eqref{WittClassOfPretzelKnot}, for every choice of an odd prime $\wp$ (compare with Theorem \ref{ConditionChecker}). The prime $\wp=2$ has no bearing here since $\det K$ is always odd, and so $\partial _2(\varphi (K))=0$ for any knot $K$.

We verify this in a case by case analysis. Let us write $p=\wp^a \alpha$, $q=\wp ^b \beta$ and $r=\wp ^c \gamma$ with each of $\alpha$, $\beta$ and $\gamma$ coprime with $\wp$, and with $a,b,c$ non-negative integers.  Similarly, write $m=\wp ^s \sigma$ so that the equation $pq+pr=-(m^2+qr)$ becomes 
\begin{equation} \label{CaseTwoEquation}
\wp ^{a}\alpha (\wp^b\beta +\wp^{c}\gamma)  = -(\wp^{2s} \sigma ^2 +  \wp ^{b+c}\beta \gamma).
\end{equation} 
In the computations to follow, we shall utilize relations R1-R3 from \eqref{WittGroupRelations}. 
\begin{itemize}
\item[1.] Case of $a$, $b$, $c$ even integers. In this case clearly $\partial _\wp (\varphi(P(p,q,r))) =0$. 
\item[2.] Case of $a$ odd and $b$, $c$ even. In this case we obtain 
$$ \partial _\wp (\varphi(P(p,q,r))) = \langle \alpha \rangle \oplus \langle \alpha \beta \gamma\rangle. $$
\begin{itemize}
\item[(i)] Subcase of $b\ge c$ and $2s\ge b+c$. Equation \eqref{CaseTwoEquation} becomes 
$$\wp ^{a+c}\alpha (\wp^{b-c}\beta +\gamma)  = -\wp^{b+c} ( \wp ^{2s-b-c} \sigma ^2 + \beta \gamma).$$ 
Comparison of powers of $\wp$ on either side leads to a contradiction unless either $\wp^{b-c}\beta +\gamma \equiv 0 \,(\text{mod } \wp)$ or $ \wp ^{2s-b-c} \sigma ^2 + \beta \gamma  \equiv 0 \,(\text{mod } \wp) $ or both. The first of these congruences can only occur if $b=c$ in which case it becomes $\beta +\gamma \equiv 0 \,(\text{mod } \wp)$, while the second is only viable if $2s=b+c$ transforming it into $\sigma ^2 + \beta \gamma  \equiv 0 \,(\text{mod } \wp)$. The former case leads to  
$$\langle \alpha \rangle \oplus \langle \alpha \beta \gamma\rangle = \langle \alpha \rangle \oplus \langle -\alpha \, \beta^2 \rangle = \langle \alpha \rangle \oplus \langle -\alpha \rangle = 0,$$
while the latter yields
$$\langle \alpha \rangle \oplus \langle \alpha \beta \gamma\rangle = \langle \alpha \rangle \oplus \langle -\alpha \, \sigma^2 \rangle = \langle \alpha \rangle \oplus \langle -\alpha \rangle = 0.$$
%
\item[(ii)] Subcase of $b\ge c$ and $2s\le b+c$. In this setup, equation \eqref{CaseTwoEquation} becomes 
$$\wp ^{a+c}\alpha (\wp^{b-c}\beta +\gamma)  = -\wp^{2s} ( \sigma ^2 + \wp^{b+c-2s}\beta \gamma).$$ 
As in the previous subcase, a comparison of powers of $\wp$ on the two sides leads quickly to either $\beta +\gamma \equiv 0 \,(\text{mod } \wp)$ or $\sigma ^2 + \beta \gamma  \equiv 0 \,(\text{mod } \wp) $, each of which implies the desired conclusion of $ \partial _\wp (\varphi(P(p,q,r))) =0$ exactly as in the previous subcase. 
\item[(iii)] Subcase of $c\ge b$. This case is treated in complete analogy to the case of $b\ge c$, as the difference in parity between $\beta$ and $\gamma$ played no role in our arguments.  
\end{itemize}
\item[3.] {\em Case of $a$, $b$ odd and $c$ even.} With these parity choices we obtain 
$$ \partial _\wp (\varphi(P(p,q,r))) = \langle \alpha \rangle \oplus  \langle \beta \rangle. $$
\begin{itemize}
\item[(i)] Subcase of $a\ge b$ and $2s\ge a+b$. In this situation equation \eqref{CaseTwoEquation} becomes 
$$ \wp ^{b+c}\gamma (\wp^{a-b}\alpha + \beta)  = -\wp^{a+b} (\wp^{2s-a-b} \sigma ^2 +  \alpha \beta ). $$
Since $b+c$ is odd and $a+b$ even, it follows that either $\wp^{a-b}\alpha + \beta \equiv 0 \,(\text{mod } \wp)$ or $\wp^{2s-a-b} \sigma ^2 +  \alpha \beta \equiv 0 \,(\text{mod } \wp)$ or both. The first relation is only possible if $a=b$ and $\alpha + \beta \equiv 0 \,(\text{mod } \wp)$ leading to 
$$\langle \alpha \rangle \oplus \langle \beta \rangle = \langle \alpha \rangle \oplus \langle -\alpha \rangle = 0. $$
The second relation is possible only if $2s=a+b$ and $\sigma ^2+ \alpha \beta \equiv 0 \,(\text{mod } \wp)$, leading to 
$$\langle \alpha \rangle \oplus \langle \beta \rangle = \langle \alpha \rangle \oplus \langle \alpha^2\beta  \rangle = \langle \alpha \rangle \oplus \langle - \alpha \sigma^2  \rangle =  \langle \alpha \rangle \oplus \langle -\alpha \rangle = 0. $$
%
\item[(ii)] {\em Subcase of $a\ge b$ and $2s\le a+b$.} Equation \eqref{CaseTwoEquation} now becomes 
$$ \wp ^{b+c}\gamma (\wp^{a-b}\alpha + \beta)  = -\wp^{2s} ( \sigma ^2 + \wp^{a+b-2s} \alpha \beta ) $$
and can only be realized if either $a=b$ and $\alpha + \beta \equiv 0 \,(\text{mod } \wp)$ or if $a+b=2s$ and $\sigma ^2+ \alpha \beta \equiv 0 \,(\text{mod } \wp)$. The conclusion $\varphi (P(p,q,r))=0$ follows as in the previous subcase. 
\item[(iii)] {\em Subcase of $a\le b$.} In the current case 3, all equations are symmetric in $a$ and $b$ and so this subcase follows from the previous two. 
\end{itemize}
We note that while the symmetry between the parameters $p$, $q$ and $r$ is broken by their different parity assumptions, these do not come into play in the above arguments. Thus case 3 also covers the situations where either only $a$ or only $b$ is even.
\item[4.] {\em Case of $a, b, c$ all odd.} In this case we obtain
$$ \partial _\wp (\varphi(P(p,q,r))) = \langle \alpha \rangle \oplus  \langle \beta \rangle \oplus \langle \gamma \rangle \oplus \langle \alpha \beta \gamma \rangle.$$
Relying on the complete symmetry of all equations under permutations of $a, b, c$, we may assume for concreteness that $a\ge b\ge c$. Thus equation \eqref{CaseTwoEquation} becomes 
\begin{equation} \label{CaseFourEquation} 
\wp^{a-c}\alpha\beta + \wp ^{a-b}\alpha \gamma + \beta \gamma = - \lambda ^2 \quad \text{ with } \quad \lambda =\frac{\wp^{s} \sigma}{\wp^{\frac{b+c}{2}}}.
\end{equation}
Notice that this equation leads to a contradiction unless $2s=b+c$. For if $2s>b+c$, a mod $\wp$ reduction gives the contradictory congruence $\beta \gamma \equiv 0\, (\text{mod } \wp)$, while $b+c>2s$ leads to the contradiction $\sigma ^2 \equiv 0\, (\text{mod }\wp)$. Thus $b+c=2s$ and so $\lambda = \sigma$. 
\begin{itemize}
\item[(i)] {\em Case of $a> b$.} In this case equation \eqref{CaseFourEquation} implies $\beta \gamma  \equiv -\sigma ^2 \, (\text{mod }\wp)$ showing that 
\begin{align*}  
\langle \alpha \rangle \oplus  \langle \beta \rangle \oplus \langle \gamma \rangle \oplus \langle \alpha \beta \gamma \rangle & =  \langle \alpha \rangle \oplus  \langle \beta \gamma^2 \rangle \oplus \langle \gamma \rangle \oplus \langle - \alpha \sigma ^2 \rangle \cr
& =  \langle \alpha \rangle \oplus  \langle -\sigma ^2 \gamma \rangle \oplus \langle \gamma \rangle \oplus \langle - \alpha  \rangle \cr
& =  \langle - \gamma \rangle \oplus \langle \gamma \rangle  \cr
& = 0.
\end{align*}
\item[(ii)] {\em Case of $a= b > c$.} Here equation \eqref{CaseFourEquation} reduces to $(\alpha +\beta)\gamma \equiv -\sigma ^2 \, (\text{mod } \wp)$ leading to 
\begin{align*}
\langle \alpha \rangle \oplus  \langle \beta \rangle \oplus \langle \gamma \rangle \oplus \langle \alpha \beta \gamma \rangle & =  \langle \alpha +\beta \rangle \oplus  \langle \alpha \beta (\alpha +\beta)  \rangle \oplus \langle \gamma \rangle \oplus \langle  \alpha \beta \gamma \rangle \cr
& =  \langle (\alpha +\beta)\gamma ^2 \rangle \oplus  \langle \alpha \beta \gamma ^2(\alpha +\beta)  \rangle \oplus \langle \gamma \rangle \oplus \langle  \alpha \beta \gamma \rangle \cr
& =  \langle -\sigma ^2\gamma \rangle \oplus  \langle -\sigma ^2\alpha \beta \gamma \rangle \oplus \langle \gamma \rangle \oplus \langle  \alpha \beta \gamma \rangle \cr
& =  \langle -\gamma \rangle \oplus  \langle -\alpha \beta \gamma \rangle \oplus \langle \gamma \rangle \oplus \langle  \alpha \beta \gamma \rangle \cr
& = 0.
\end{align*}
\item[(iii)] {\em Case of $a=b=c$.} Here we are led to the equation $(\alpha +\beta )\gamma = - (\alpha \beta +  \sigma ^2)$. A repeated use of this in the next lines gives the claimed result:
\begin{align*}
\langle \alpha \rangle \oplus  \langle \beta \rangle \oplus \langle \gamma \rangle \oplus \langle \alpha \beta \gamma \rangle & =  \langle \alpha +\beta \rangle \oplus  \langle \alpha \beta (\alpha +\beta)  \rangle \oplus \langle \gamma \rangle \oplus \langle  \alpha \beta \gamma \rangle \cr
& =  \langle (\alpha +\beta)\gamma ^2 \rangle  \oplus \langle  \alpha \beta \gamma \rangle \oplus  \langle \alpha \beta(\alpha +\beta)  \rangle \oplus \langle \gamma \rangle \cr
& =  \langle -\alpha \beta \gamma -\sigma ^2\gamma \rangle \oplus \langle  \alpha \beta \gamma \rangle \oplus  \langle \alpha \beta(\alpha +\beta)  \rangle \oplus \langle \gamma \rangle  \cr
& =  \langle -\sigma ^2 \gamma \rangle \oplus  \langle \sigma ^2 \alpha \beta \gamma ^3(\alpha \beta +\sigma^2) \rangle \oplus  \langle \alpha \beta(\alpha +\beta)  \rangle \oplus \langle \gamma \rangle \cr
& = \langle - \gamma \rangle \oplus  \langle  -\alpha \beta \gamma^2(\alpha  +\beta) \rangle \oplus  \langle \alpha \beta(\alpha +\beta)  \rangle \oplus \langle \gamma \rangle \cr
& =   \langle  -\alpha \beta (\alpha  +\beta) \rangle \oplus  \langle \alpha \beta(\alpha +\beta)  \rangle  \cr
& = 0.
\end{align*}
\end{itemize}
\end{itemize}
The above calculations show that Theorem \ref{main} applies to the symmetrized linking form of every pretzel knot $P(p,q,r)$ with $pq+pr+qr=-m^2$ with $m\in \mathbb N$ odd. Accordingly, the signature of any such knot is a multiple of $8$. On the other hand, the signature can be computed as \cite{jabuka1} 
$$\sigma (P(p,q,r)) = -(p+q)  +Sign(p)+Sign(q) -Sign(pq(p+q))+Sign((p+q)(pq+pr+qr)).$$
Since $pq+pr+qr=-m^2$, the above simplifies to 
$$\sigma (P(p,q,r)) = -(p+q)  +Sign(p)+Sign(q) -Sign(pq(p+q))-Sign(p+q).$$
Regardless the signs of $p$ and $q$, the above equation always reduces to  $\sigma (P(p,q,r)) = -(p+q)$, verifying the claim from Proposition \ref{DiophantineTheorem}. 
%
\bibliography{Bibliography.bib}{}
\bibliographystyle{plain}

\end{document}